\documentclass[12pt]{amsart}

\usepackage{amsmath,amssymb,latexsym} 
\usepackage{fullpage}
\usepackage{etoolbox}
\usepackage{enumitem}
\usepackage{color}
\usepackage[colorlinks,linkcolor=blue,citecolor=magenta]{hyperref}
\usepackage[colorinlistoftodos,bordercolor=orange,backgroundcolor=orange!20,linecolor=orange,textsize=scriptsize]{todonotes}
\usepackage{soul}
\usepackage{microtype}
\usepackage{yhmath}

\newtheorem{theorem}{Theorem} 
\newtheorem{corollary}{Corollary}

\newtheorem{lemma}{Lemma}

\newtheorem*{proof-claim}{Proof}

\theoremstyle{definition}
\newtheorem*{remark}{Remark}

\newtheorem{question}{Question}

\newenvironment{changemargin}[2]{\begin{list}{}{%
\setlength{\topsep}{0pt}%
\setlength{\leftmargin}{0pt}%
\setlength{\rightmargin}{0pt}%
\setlength{\listparindent}{\parindent}%
\setlength{\itemindent}{\parindent}%
\setlength{\parsep}{0pt plus 1pt}%
\addtolength{\leftmargin}{#1}%
\addtolength{\rightmargin}{#2}%
}\item }{\end{list}}

\AtBeginEnvironment{proof-claim}{\vspace{-0.5cm}\footnotesize\begin{changemargin}{0.5cm}{0.5cm}}
\AtEndEnvironment{proof-claim}{\end{changemargin}}

\def\leq{\leqslant}
\def\geq{\geqslant}

\def\P{\mathcal{P}}
\def\ST{\operatorname{ST}}

\begin{document} 

\title{On the number of star-shaped classes in optimal colorings of Kneser graphs}

\author{Hamid Reza Daneshpajouh}
\address{Hamid Reza Daneshpajouh,
University of Nottingham Ningbo China, 199 Taikang E Rd, Yinzhou, Ningbo, Zhejiang, China, 315104}
\email{Hamid-Reza.Daneshpajouh@nottingham.edu.cn}



\begin{abstract}
A family of sets is called star-shaped if all the members of the family have a point in common. The main aim of this paper is to provide a negative answer to the following question
raised by James Aisenberg et al [\textit{Short proofs of the kneser–Lov\'asz coloring principle, Information and Computation, 261:296–310, 2018.}], for the case $k=2$.
\vspace{0.2cm}

\textbf{Question.}
 Do there exist $(n-2k+2)$-colorings of the $(n , k)$-Kneser graphs with more than $k-1$ many non-starshaped
color classes?
\vspace{0.2cm}


\end{abstract}


\maketitle 

\section*{Introduction}\label{sec:intro}
In this paper all graphs are finite, undirected and simple.  A (proper) coloring of a graph is an assignment of colors to its vertices such that no two adjacent vertices have the same color. The chromatic number $\chi(H)$ of a graph $H$ is the smallest number of colors in a coloring of $H$. For a given coloring of a graph $H$, a subgraph of $H$ is called colorful if all its vertices receives different colors. The $(n , k)$-Kneser graph $KG(n,k)$ is a graph whose vertices correspond to all $k$-subsets of the set $[n]=\{1,2\ldots, n\}$, and where two vertices are adjacent if and only if the two corresponding sets are disjoint. Finally, a family $\mathcal{F}$ of sets is called star-shaped if $\bigcap\mathcal{F}\neq\emptyset$.

Almost seventy years ago, a purely combinatorial conjecture was raised by Kneser which in the language of graph theory is equivalent to say $\chi(KG(n,k))=n-2k+2$ for $n\geq 2k-1$. About two decades later, Lov\'asz~\cite{Lo78} proved Kneser's conjecture by taking an extraordinary approach, using tools from algebraic topology. Since then, many attentions have been drawn to study various problems related to this conjecture, including a large number of new proofs such as~\cite{daneshpajouh2018new, Do88, greene2002new, matouvsek2004combinatorial}, and many generalizations such as ~\cite{alishahi2015chromatic, alon1986chromatic, frick2017intersection, Me11, daneshpajouh2019chromatic, daneshpajouh2021neighborhood, daneshpajouh2018hedetniemi}. In particular, James Aisenberg et al~\cite{aisenberg2018short} gave a new proof of this conjecture uses a simple counting argument based on the Hilton–Milner theorem, for all but finitely many cases\footnote{Here, we assume that $k$ is a fixed number.}.

Indeed, their main idea was based on that fact that if there was an $(n-2k+1)$-coloring of $KG(n,k)$, then it would contain a star-shaped color class~\cite[Lemma 8]{aisenberg2018short} provided that $n$ is enough large\footnote{For instance $n\geq k^4$ works.}. Then removing this vertex would lead to a  $((n-1)-2k+1)$-coloring of $KG(n-1,k)$. So, to verify the conjecture one just needs to check the validity of the conjecture for finitely many base cases which can be checked, for instance, using the topological arguments provided by Lov\'asz for these cases\footnote{If $k$ is small enough ($k=2$, or $3$), then the base cases can be checked by hand or a computer. However, for $k\geq 4$ we do not know whether there exist a ``purely combinatorial way" to establish the conjecture.}.

Actually, they showed if there was such a coloring then it  would contain many star-shaped color classes~\cite[Lemma 9]{aisenberg2018short}, again provided $n$ is enough large. Motivated by this result, they showed that there is an optimal coloring ( a coloring with $(n-2k+2)$ colors) of $KG(n,k)$ with $(k-1)$ star-shaped color classes provided that $n\geq 3k-3$ and then they raised the following question in this regard.
\begin{question}[\cite{aisenberg2018short}]
Do there exist $(n-2k+2)$-colorings of the $(n , k)$-Kneser graphs with more than $k-1$ many non star-shaped
color classes?
\end{question}
The main objective of this paper is to give a negative answer to the above question for the case $k=2$. As a corollary, this result leads to a purely combinatorial proof for the $\mathcal{K}_{l,m}$-Theorem~\cite[Theorem 2]{simonyi2007colorful} for the case $KG(n,2)$. More precisely,  for every non-negative integers $k, l$ with $k+l+m=n-2$, and any optimal proper coloring of $KG(n,2)$, there is a colorful complete bipartite  $\mathcal{K}_{l, m}$ subgraph in $KG(n,2)$. A similar result for existence of colorful complete tripartite sungraph in $KG(n,2)$ will be provided as as well.


\section{Main Results}\label{sec:prel}
First note that we can view $KG(n,2)$ as the complement of the line graph of the complete graph $\mathcal{K}_n$ and therefore proper colorings of $KG(n,2)$ are exactly partitions of the edge set of the complete graphs $\mathcal{K}_n$ into stars and triangles. More generally, as mentioned here~\cite{daneshpajouh2021colorings}, if $G$ is the complement of the line graph of a graph $H$, then proper colorings of $G$ are exactly partitions of the edge set of $H$ into stars and triangles. A {\em star} is a tree with a vertex, the {\em center} of the star, connected to all other vertices. A single-edge tree is in particular a star. In this particular case, we will always assume that exactly one of the two vertices has been identified as the center, so as to be in a position of always speak of the center of a star without ambiguity. A {\em triangle} is a circuit of length $3$. We call such a partition of $H$ into stars and triangles an {\em $\ST$-partition}. 

\begin{lemma}\label{lem:min-tri}
Consider an optimal $\ST$-partition $\P$ of a graph $H$. Then, at most one vertex of each triangle in $\P$ can be the center of a star.
\end{lemma}

\begin{proof}
We proceed by a contradiction. Suppose that $T$ is a triangle in $\P$ with at least two vertices $x$, $y$ such that each of them is also the center of a star, name it $S_x$, and $S_y$ respectively. Then we can add the two edges of $T$ incident to $x$ to $S_x$, and the remaining edge of $T$ to $S_y$, and finally remove $T$ from $S$. This leads to a new $\ST$-partition with fewer elements which contradicts the optimality of $\P$.
\end{proof}

\begin{lemma}\label{lem:min-tri1}
Consider an optimal $\ST$-partition $\P$ of a graph $H$. If there is a triangle $T$ in $\P$, then each vertex $x$ of $H$ which does not belong to any triangle of $\P$ and it is connected to at least two vertices of $T$ is the center of a star.
\end{lemma}
\begin{proof}
We will use a proof by contradiction. Suppose $x$ is not the center of any star. If $a, b$ are two vertices of $T$ which are connected to $x$, then each of the edges $\{a,x\}$ and $\{b,x\}$ must belong to a star or a triangle in $\P$. These imply that $a$ and $b$ are both must be the center of stars, as no triangle in $\P$ having $x$ as its vertices, which contradicts the previous lemma.  
\end{proof}

\begin{theorem}\label{The:main}
There is no optimal ST-partition of the complete graph $\mathcal{K}_n$ with more than one triangle. Moreover, if $n\geq 3$, any optimal ST-partition of $\mathcal{K}_n$ contains exactly one triangle.
\end{theorem}
\begin{proof}

Suppose the contrary and choose the minimum integer $n$ with the property that there is an optimal ST-partition $\P$ of the complete graph $\mathcal{K}_n$ with more than one triangle. Let $T_1, \ldots, T_k$ be the list of all the triangles in $\P$ where $k\geq 2$. Let $A=\bigcup_{i=1}^{k}V(T_i)$ be the set of the vertices of these triangles, $B\subseteq [n]\setminus A$ be the set of the rest vertices which are centers of stars, and finally put $C=[n]\setminus (A\cup B)$. By Lemma~\ref{lem:min-tri1}, we have $C=\emptyset$. Moreover, Lemma~\ref{lem:min-tri} implies $B=\emptyset$. Indeed, if there was a vertex of $\mathcal{K}_n$ in $B$, then removing this vertex leads to an optimal ST-partition with at least two triangles for $\mathcal{K}_{n-1}$, which would contradict the minimality of $n$, and therefore we must have $B=\emptyset$. In summery, we have shown that any vertex of $\mathcal{K}_n$ belongs to at least one of these triangles $T_1, \ldots, T_k$. Now, we consider two cases:

\vspace{0.5cm}
\textbf{Case I.}
First, we consider the case that there are exactly two triangles in $\P$, say $T_1$, and $T_2$. We have two possibilities.

\begin{enumerate}
    \item  $T_1$, and $T_2$ are vertex disjoint
    
\begin{center}
    \includegraphics[scale=0.5]{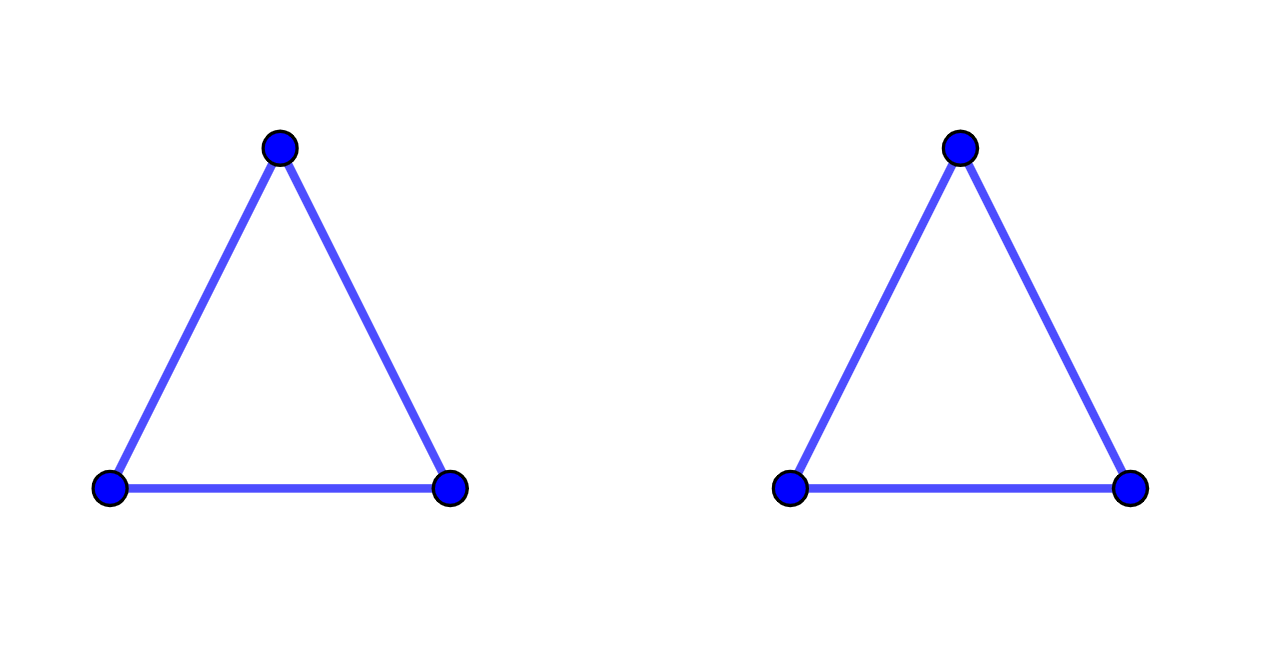}
\end{center}
    
    Note that in this case we have $n=6$ and therefore $|\P|=4$ (since $\P$ is an optimal ST-partition of $\mathcal{K}_6$). But on the other hand,
    we have
$$|P|\geq  2 + 3= 5$$
as $T_1, T_2\in P$ and at least three vertices of $T_1\cup T_2$ are centers. To see the latter claim, consider the edges $e_1, e_2, e_3$ of $\mathcal{K}_6$ which are shown in below 
\begin{center}
    \includegraphics[scale=0.5]{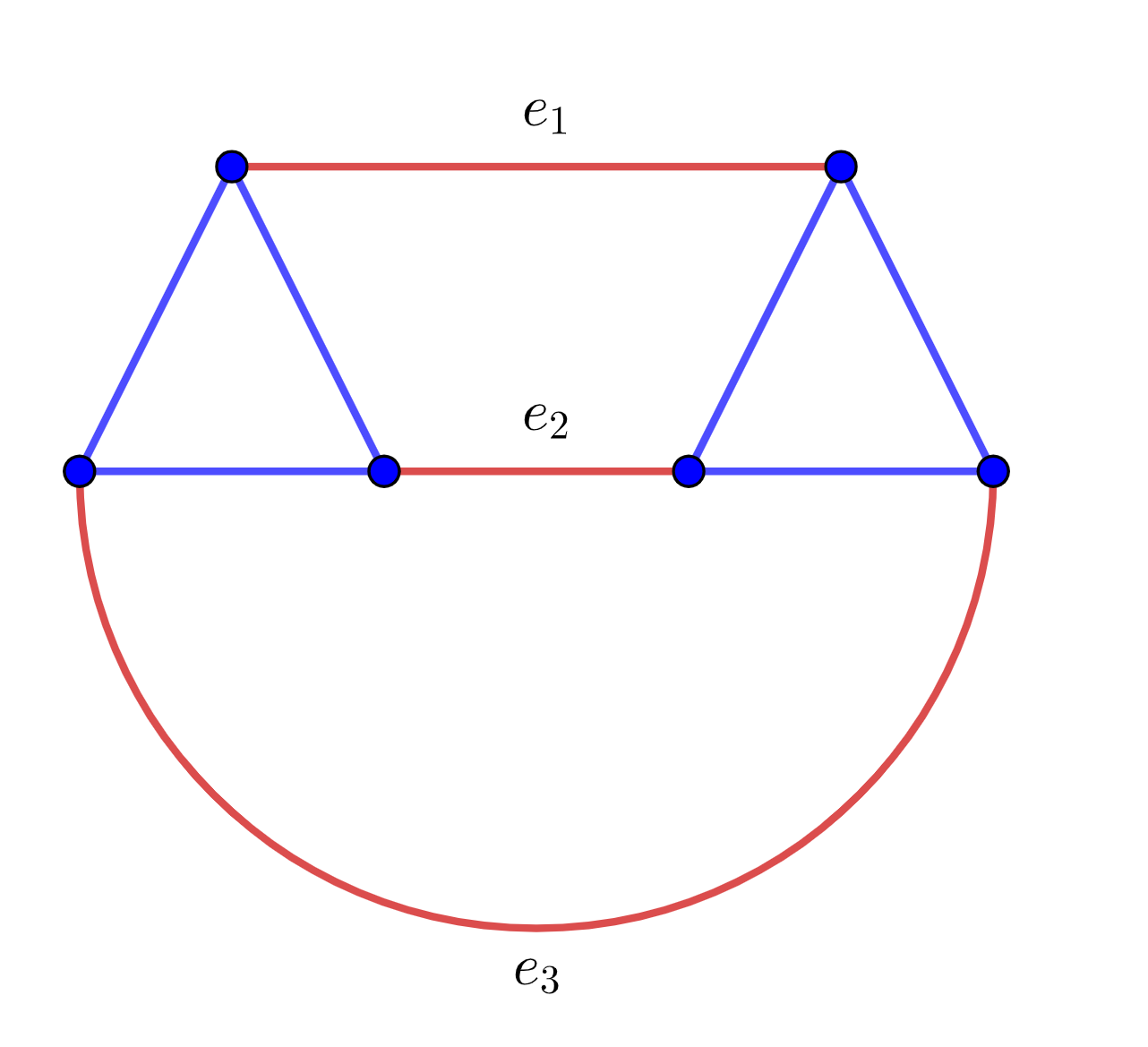}
\end{center}
Each of these edges must belong to a triangle or a start in $\P$. But none of them can belong to a triangle as there are no other triangle in $\P$ except $T_1$ and $T_2$. So, each of them must belong to a star, and hence one of the endpoints of each of $e_1, e_2,$ and $e_3$ must be the center of a star. So, there are at least three stars in $\P$ as these edges are pairwise disjoint. Hence $\P$ must have at least $5$ elements which contradicts the optimality of $\P$.
\vspace{0.3cm}
\item 
If $T_1$, and $T_2$ have a vertex in common,

\begin{center}
    \includegraphics[scale=0.5]{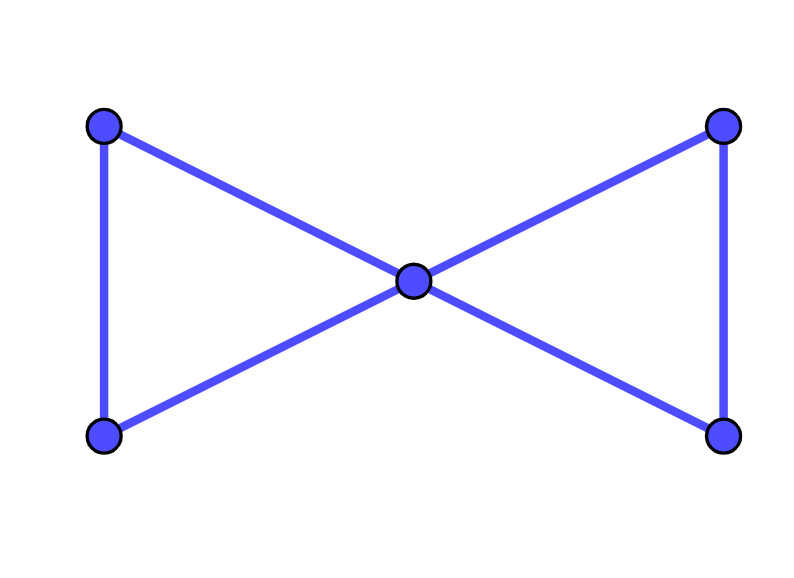}
\end{center}

then we have
$$|P|\geq  2 + 2= 4,$$
as $T_1, T_2\in P$, and at least two vertices of $T_1\cup T_2$ are centers. Again to see why the latter claim is true, it is enough to consider the edges $e_1$ and $e_2$ which are depicted below 
\begin{center}
    \includegraphics[scale=0.5]{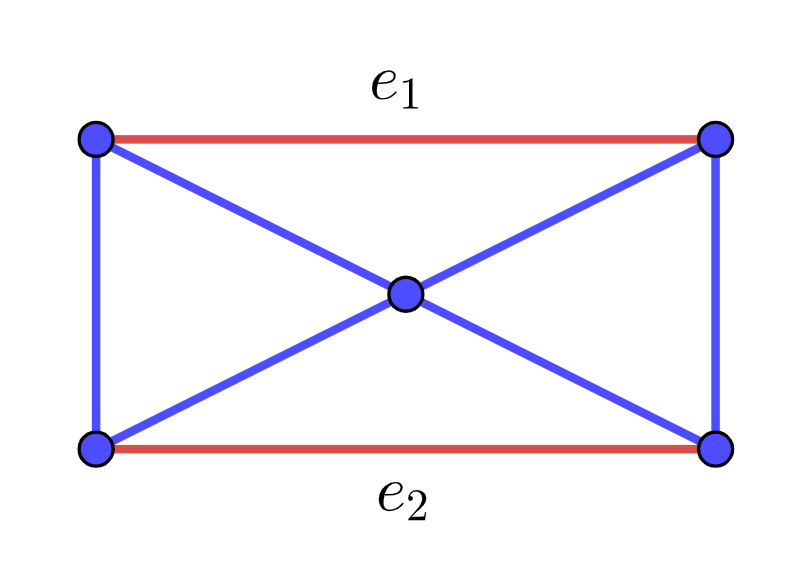}
\end{center}
and repeat the same argument. So, the size of $\P$ is at least four. But,  this is again impossible as $P$ is an optimal ST-partition of $\mathcal{K}_5$ which implies $|\P|=3$.
\end{enumerate}

\vspace{.5cm}
\textbf{Case II.} 
Finally, suppose there are at least 3 triangles in $\P$. Then, this implies that no vertex of a triangle is a center of a star. Indeed, if there was a triangle $T$ in $\P$ which one of its vertices was a center of a star, then all the other triangles in $\P$ except at most one of them is attached to this triangle at $x$ (otherwise removing this vertex would lead to an optimal ST-partition for $\mathcal{K}_{n-1}$ with at least two triangles which would contradict the minimality of $n$) as shown in below
\begin{center}
    \includegraphics[scale=0.35]{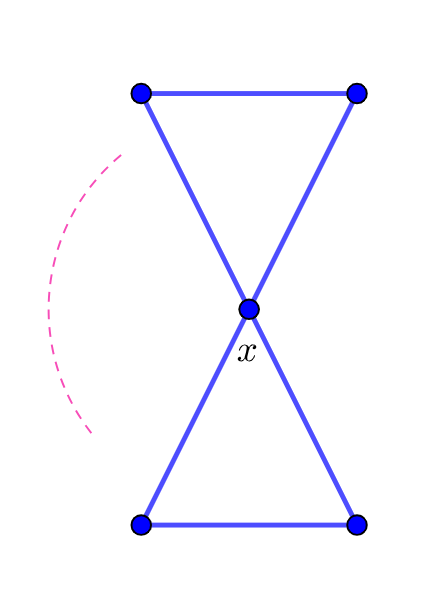}
\end{center}
Now, consider the edges $e_1$ and $e_2$ of $\mathcal{K}_n$ as shown below.

\begin{center}
    \includegraphics[scale=0.35]{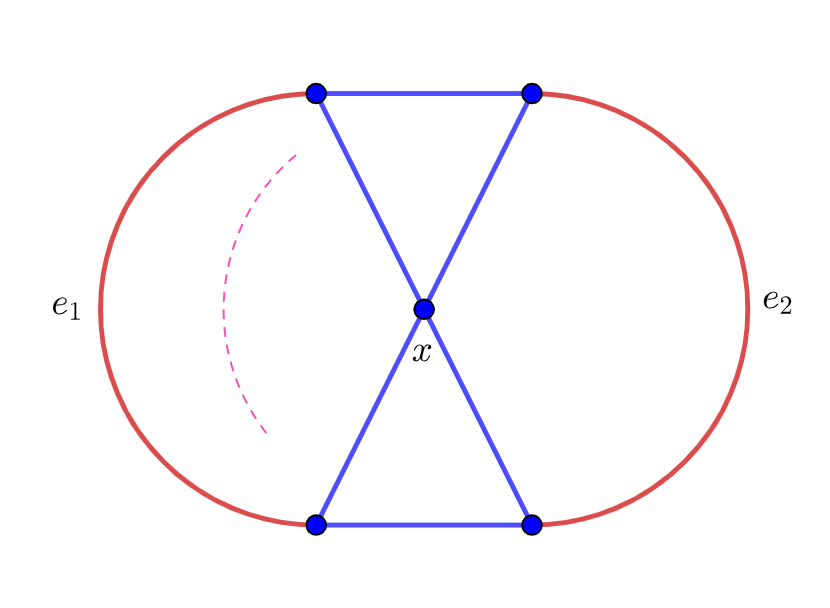}
\end{center}

At most one of these edges can be covered with that only one possible triangle $\P$ which is not attached to $x$. So, at least one of the end points of $e_1$ or $e_2$ must be the center of a star. But then, we would have a triangle with two of its vertices ($x$ and of the endpoints of $e_1$ or $x$ and of the endpoints of $e_2$ ) are the center of stars which would contradict Lemma~\ref{lem:min-tri}. So, all member of $\P$ are triangles. Thus, we must have
$$3(n-2)=\binom{n}{2},$$
as $\P$ is an optimal ST-partition of edges of $\mathcal{K}_n$. This equation just has two solutions $n=3$ or $4$, but on the other hand having at least 3 triangles in $\P$ implies $n\geq 6 $ which is again a contradiction.

So, we have shown that any optimal ST-partition of $\mathcal{K}_n$ contains at most one triangle. Moreover, we show that every optimal ST-partition of $\mathcal{K}_n$ contains exactly one triangle provided that $n\geq 3$. Indeed, the size of any ST-partition $\P$ of $\mathcal{K}_n$ without a triangle is at least $n-1$ as at least one of the endpoints of each edge must be the center of a star. Thus, any optimal ST-partition $\P$ of $\mathcal{K}_n$ where $n\geq 3$ has exactly one triangle (as $|\P|=n-2$ for $n\geq 3$).


\end{proof}

\begin{corollary}
Let $n\geq 3$ be a natural number. For every non-negative integers $l$ and $m$ with $l+m=n-2$, and any optimal proper coloring of $KG(n,2)$, there is a colorful complete tripartite  $\mathcal{K}_{l, m}$ subgraph in $KG(n,2)$.  
\end{corollary}
\begin{proof}
The assertion is obviously true for $n=3$. So, without loss of generality we can assume $n\geq 4$. Let $\P$ be an arbitrary optimal ST-partition of $\mathcal{K}_n$. By Theorem~\ref{The:main}, $\P$ contains exactly one triangle, say $T$ with vertex set $\{1,2,3\}$. Moreover, no vertex of this triangle is a center of a star. Since otherwise, we could delete this vertex and find a minimal ST-partition for $\mathcal{K}_{n-1}$ without any triangle. Then, this would imply that $n-1\leq 2$, which contradicts our assumption. Now, $\mathcal{K}_{A, B}$ is the desired subgraph in $KG(n,2)$ where 
\begin{align*}
 & A=\{\{1, i\}: 4\leq i\leq l+3\}\\
 & B=\{\{2, j\}: l+4\leq j\leq n\}\cup\{\{2,3\}\}
\end{align*}
as each $i\neq 1, 2, 3$ is a center of a star and no vertex of $T$ is a center.  
\end{proof}
Above proof reveals that we can have a similar result about the existence of colorful complete tripartite subgraph in $KG(n,2)$ as well.
\begin{corollary}
Let $n\geq 6$ be a natural number. For every integers $k, l, m\geq 1$ with $k+l+m=n-3$, and any optimal proper coloring of $KG(n,2)$, there is a colorful complete tripartite  $\mathcal{K}_{k,l, m}$ subgraph in $KG(n,2)$.  
\end{corollary}
\begin{proof}
Again, by Theorem~\ref{The:main}, $\P$ contains exactly one triangle, say $T$ with vertex set $\{1,2,3\}$ and also as mentioned before no vertex of this triangle can be a center of a star. Now, $\mathcal{K}_{A, B, C}$ is the desired subgraph in $KG(n,2)$ where 
\begin{align*}
 & A=\{\{1, i\}: 4\leq i\leq k+3\}\\
 & B=\{\{2, i\}: k+4\leq i\leq k+l+3\}\\
 & C=\{\{3, i\}: k+l+4\leq i\leq n\}
\end{align*}
as each $i\neq 1, 2, 3$ is a center of a star and no vertex of $T$ is a center.  
\end{proof}
\begin{remark}
Note that in the previous corollary the assumption $k+l+m=n-3$ cannot be replaced with $k+l+m=n-2$. To see this, consider the following coloring of $KG(n,2)$
\begin{equation*}
c(\{i < j\})=\begin{cases}
          3 \quad &\text{if} \, j\leq 3 \\
          j \quad &\text{if} \,  i\leq 3\, \&\, j\geq 4   \\
          i \quad &\text{if} \,  i\geq 4\, \\
     \end{cases}
\end{equation*}
We can see this coloring as a coloring of edges of the complete graph $\mathcal{K}_n$. With having this interpretation in mind, one can check that there is no colorful cycle in $\mathcal{K}_n$, i.e., a cycle whose all edges receives different colors. Now, if there was a colorful complete tripartite subgraph $\mathcal{K}_{A,B,C}$ in $KG(n,2)$ where $|A|=k, |B|=l, |C|=m$ with $k,l,m\geq 1$ and $m+l+k= n-2$ , then 
, as none of $A,B$, and $C$ form a cycle in $K_n$, we must have $$|\bigcup A|\geq k+1, |\bigcup B|\geq l+1, |\bigcup C|\geq m+1.$$ Now, since the sets $\bigcup A, \bigcup B , \bigcup C$ are pairwise disjoint, we must have $(k+1)+(l+1)+(m+1)\leq n$
which implies $m+l+k\leq n-3$. Then, this would contradict our assumption that $m+l+k= n-2$.
\end{remark}

\subsection*{Acknowledgments} 
The idea of this paper was inspired while the author was at Universit\'e Paris Est as a postdoctoral researcher under the supervision of Fr\'ed\'eric Meunier. The author would like to express his deepest gratitude to him for his invaluable advice, continuous support, his generous hospitality, and his valuable comments on the first draft of this paper. Also a part of this work was done when the author was at the School of Mathematics of IPM as a guest researcher. Special thanks go to IPM as in that period of time the research of the author was supported by a grant from IPM. 

\bibliographystyle{plain}
\bibliography{main}

\end{document}